\documentclass[hidelinks, 11pt]{article}
\usepackage[a4paper,bindingoffset=0.in,left=2.5cm,right=2.5cm,top=2cm,bottom=3cm]{geometry}
\usepackage{amsmath, amssymb,amsthm}
\usepackage{graphicx,subfigure}
\usepackage{url}
\usepackage{enumitem}
\usepackage[utf8]{inputenc}
\usepackage{lmodern}
\usepackage{setspace}
\usepackage[normalem]{ulem}

\usepackage{wrapfig}
\usepackage{mdframed}
\usepackage[a]{esvect}
\usepackage{color}

\newtheorem{theorem}{Theorem}[section]
\newtheorem{proposition}[theorem]{Proposition}

\newtheorem{lemma}[theorem]{Lemma}

\newtheorem{definition}[theorem]{Definition}

\theoremstyle{definition} 
\newtheorem{definition}[theorem]{Definition}

\newtheorem{example}[theorem]{Example}

\theoremstyle{theorem}
\newtheorem{mainresult}{Theorem}[section]

\usepackage{mathtools, amsmath, amssymb, graphicx, amsfonts}
\usepackage{xcolor} 
\usepackage[colorlinks=true,allcolors=blue!75!black,backref=page]{hyperref} 
\usepackage{bbm} 
\usepackage{float} 
\usepackage{amsopn} 
\usepackage{mathrsfs} 
\usepackage{comment} 
\usepackage{hyperref} 
\usepackage{amsthm} 
\renewenvironment{proof}{{\noindent\bfseries Proof.}}{\qed} 

\renewcommand{\ker}{\operatorname{ker}}
\newcommand{\diag}{\operatorname{diag}}

\newcommand{\spn}{\operatorname{span}}
\newcommand{\kem}{K}
\newcommand{\R}{\mathbb{R}}

\begin{document}

\title{Kemeny's constant and the Lemoine point of a simplex}

\author{
  Karel Devriendt\footnote{\textbf{Contact:} \href{mailto:karel.devriendt@mis.mpg.de}{karel.devriendt@mis.mpg.de}}\\
  \small{Max Planck Institute for Mathematics in the Sciences, Leipzig, DE}
}

\date{\small\textit{Dedicated to Piet Van Mieghem, on the occasion of his 60th birthday}}

\maketitle
\begin{abstract}
Kemeny's constant is an invariant of discrete-time Markov chains, equal to the expected number of steps between two states sampled from the stationary distribution. It appears in applications as a concise characterization of the mixing properties of a Markov chain and has many alternative definitions. In this short article, we derive a new geometric expression for Kemeny's constant, which involves the distance between two points in a simplex associated to the Markov chain: the circumcenter and the Lemoine point. Our proof uses an expression due to Wang, Dubbeldam and Van Mieghem of Kemeny's constant in terms of effective resistances and Fiedler's interpretation of effective resistances as edge lengths of a simplex.
\end{abstract}


\section{Introduction}
A discrete-time Markov chain is a sequence $X_0,X_1,\dots$ of random variables on a state space $\mathcal{X}$, where for all $k>0$ the distribution of $X_{k+1}$ only depends on $X_{k}$. Under suitable conditions, the distribution of $X_k$ converges to a distribution $\pi$, the  stationary distribution. Pick an arbitrary state $x$ and let $y$ be a random state sampled from $\pi$. Then the average number of steps it takes the Markov chain starting at state $x$ to reach state $y$ is independent of $y$ and is called Kemeny's constant (see Theorem \ref{th: Kemeny's constant is a constant}). 

Different explanations have been offered for why Kemeny's constant is a constant: based on the maximum principle \cite{doyle_2009_kemeny}, on the combinatorics of spanning forests \cite{kirkland_2021_directed} and on potential theory and random walks \cite{bini_2018_why}. At the same time, many alternative expressions have been found for Kemeny's constant, for instance based on the spectrum or group inverse of matrices associated with the Markov chain \cite{levene_2002_surfer, hunter_2006_mixing, catral_2010_kemeny, carmona_2023_mean} or based on the effective resistance \cite{wang_2017_kemeny}. Moreover, following its interpretation as an invariant that reflects the average mixing behaviour of a Markov chain, there have been several applications of Kemeny's constant \cite{kooij_2020_kemeny, altafini_2023_centrality}. For more on the history, theory and applications of Kemeny's constant, we refer to \cite{breen_2018_markov}.
\\~\\
Wang, Dubbeldam and Van Mieghem showed that for a certain class of Markov chains --- those corresponding to random walks on undirected graphs --- Kemeny's constant can be expressed in terms of the commute time\footnote{More precisely, they showed that it can be expressed in terms of the effective resistance, which is proportional to the commute time.}. For two states $x$ and $y$, the commute time $c_{xy}$ is the average number of steps it takes the Markov chain starting at $x$ to reach $y$ and then $x$ again. As the commute time of a random walk on a graph is proportional to the so-called effective resistance \cite{chandra_1996_electrical}, Fiedler's geometric theory of effective resistances \cite{fiedler_2011_matrices} leads to the following: for each (suitable) Markov chain, there exists a simplex $S$ whose vertices are identified with states in $\mathcal{X}$ and whose squared edge lengths correspond to commute times between states (see Theorem \ref{th: simplex embedding theorem}). Our main result is related to the following two points associated with this simplex $S$: the \textit{circumcenter} is the unique point at equal distance from all vertices, and the \textit{Lemoine point} is the unique point whose total squared distance to all facets (codimension-one faces) of $S$ is minimal; see Proposition \ref{proposition: circumcenter and Lemoine point coordinates} for an algebraic characterization. We prove the following expression of Kemeny's constant in terms of these geometric quantities.
\begin{mainresult}\label{th: main theorem-intro}
In a finite, irreducible, aperiodic, reversible, loop-free Markov chain, Kemeny's constant $\kem$ is equal to
$$
\kem = R^2 - \Vert \gamma - \ell\Vert^2_2,
$$
where $\gamma$ is the circumcenter at distance $R$ from each vertex of $S$, the simplex whose squared edge lengths are given by the commute times of the Markov chain, and $\ell$ is the Lemoine point of $S$.
\end{mainresult}
\begin{example}
Consider the three-state Markov chain in the left figure below.
\begin{figure}[h!]
    \centering    \includegraphics[width=0.65\textwidth]{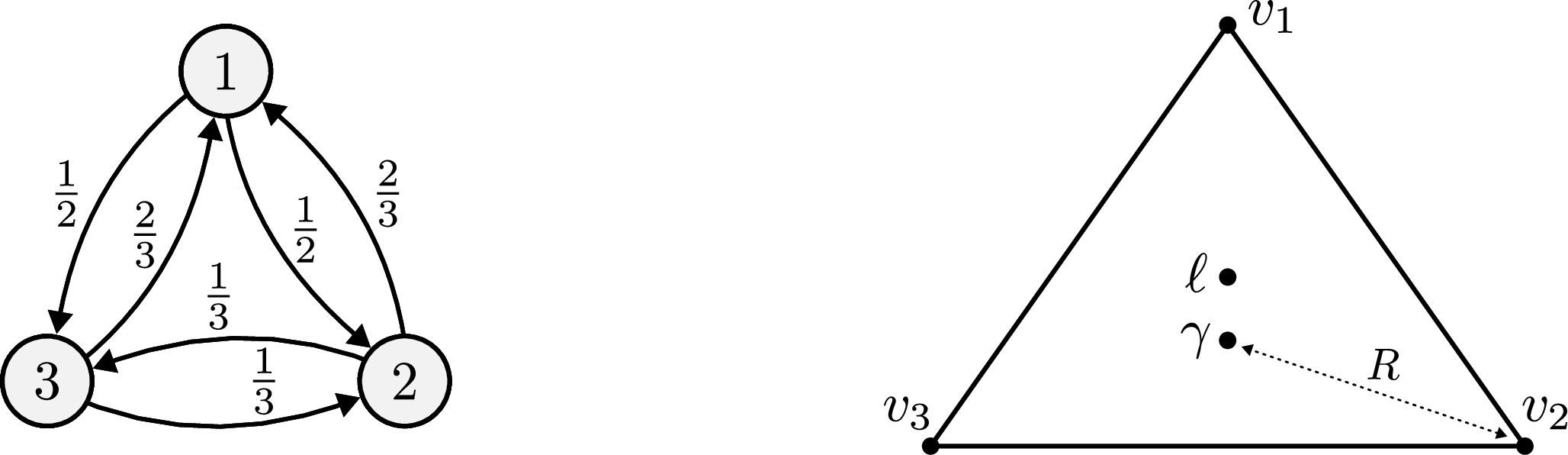}
\end{figure}

The stationary distribution of the Markov chain is $\pi=(\tfrac{4}{10},\tfrac{3}{10},\tfrac{3}{10})$ and the commute times are $c_{12}=3.75, c_{23}=5$ and $c_{13}=3.75$. These commute times are the squared edge lengths of the triangle on the right: $\Vert v_1-v_2\Vert_2=\Vert v_1-v_3\Vert_2=\sqrt{3.75}$ and $\Vert v_2-v_3\Vert_2 = \sqrt{5}$. The circumcenter $\gamma$ and Lemoine point $\ell$ of the triangle are at distance $\Vert\gamma-\ell\Vert_2=0.2372$ and the circumradius is $R=1.1859$. By Theorem \ref{th: main theorem-intro}, Kemeny's constant is $1.1859^2 - 0.2372^2=1.35$. This agrees with other definitions, see for instance equation \eqref{eq: definition Kemeny's constant via commute times} in Section \ref{section: markov chains and kemeny}.
\end{example}
The proof of Theorem \ref{th: main theorem-intro} is given in Section \ref{section: special points and main result} and amounts to an application of the powerful tools and the dictionary between simplices and graphs --- and by extension certain Markov chains --- set up by Fiedler in \cite{fiedler_2011_matrices} and further developed by the author in \cite{devriendt_2019_simplex, devriendt_2022_effective, devriendt_2022_thesis}. The conditions on the Markov chain are quite restrictive, but many of these conditions are necessary to define Kemeny's constant. Moreover, the conditions of Theorem \ref{th: main theorem-intro} are satisfied by random walks on finite connected undirected graphs, which is an important class of Markov chains.

To avoid confusion, we note that in contrast to other works on Kemeny's constant, Theorem \ref{th: main theorem-intro} and the rest of this article are no attempt at explaining or providing intuition why Kemeny's constant is a constant. We simply prove a new geometric expression for Kemeny's constant. An implicit suggestion, however, is that Fiedler's simplex geometry might be a useful tool for problems in Markov chain theory. For more on this connection, see \cite{fiedler_2011_matrices} or \cite{devriendt_2022_thesis}.
\\~\\
\textbf{Organization:} The rest of this article is organized as follows. Section \ref{section: markov chains and kemeny} introduces Markov chains and Kemeny's constant. In Section \ref{section: simplex from M} we introduce the simplex associated with a Markov chain (Theorem \ref{th: simplex embedding theorem}). Section \ref{section: special points and main result} then introduces the points of interest in this simplex (Proposition \ref{proposition: circumcenter and Lemoine point coordinates}) and proves the main result (Theorem \ref{th: main theorem} (=\ref{th: main theorem-intro})). 
\\
~
\\
\textbf{Acknowledgements:} The work that led to this article was started during a research visit at the MAPTHE group at UPC Barcelona. I want to thank \'{A}ngeles, Andr\'{e}s, Maria Jos\'{e}, Enric and Margarida for their warm hospitality and inspiring discussions. The visit was made possible by support from the Alan Turing Institute under the
EPSRC grant EP/N510129/1.

\section{Markov chains, Kemeny's constant and commute times}\label{section: markov chains and kemeny}
A discrete-time Markov chain $M$ is a sequence $(X_n)_{n\in \mathbb{N}}$ of random variables on a state space $\mathcal{X}$, where the distribution of $X_{k+1}$ only depends on $X_k$. This distribution can be characterized by the transition probabilities
$$
P_{xy} := \Pr[X_{k+1} = y \,\vert\, X_{k}=x] \text{~for all $x,y\in \mathcal{X}$}.
$$
We will assume that $\vert \mathcal{X}\vert$ is finite which means that $P$ is a matrix, the transition matrix. A finite Markov chain is fully determined by its state space, the distribution $\pi_0$ of its initial state and its transition matrix and we may thus write $M=(\mathcal{X},P,\pi_0)$. We will assume that $M$ is irreducible and aperiodic: there exists a finite $t_0>0$ such that for all integers $t\geq t_0$, the matrix $P^t$ has all entries positive. Under these assumptions, the distribution of $X_n$ converges to a unique distribution $\pi$, the stationary distribution \cite{Levin_2009_markov}.

For two states $x$ and $y$, the hitting time $h_{xy}$ from $x$ to $y$ is the expected number of steps of the Markov chain starting at $x$ to reach $y$; more precisely
$$
h_{xy} := \mathbb{E}\left(\min\left\lbrace t \in\mathbb{N} \,:\, X_0=x \text{~and~}X_t = y\right\rbrace\right)
$$
with expectation over realizations of the Markov chain. We note that $h_{xx}=0$ for every state $x$. For a fixed state $x$ and a random state $Y$ with distribution $\pi$, the expected hitting time from $x$ to $Y$ is equal to
$$
K_x := \mathbb{E}(h_{xY}) = \sum_{y\in \mathcal{X}}\pi(y)h_{xy}
$$
where the expectation is over realizations of the random state $Y$. The interest in the quantity $K_x$ stems from the observation, first noted in \cite{Kemeny_1960_finite}, that $K_x$ is independent of $x$.
\begin{theorem}\label{th: Kemeny's constant is a constant}
In a finite, irreducible, aperiodic Markov chain $M$ with stationary distribution $\pi$ and hitting times $h$, the value $K_x = \sum_{y\in\mathcal{X}}\pi(y)h_{xy}$ is independent of $x$. This common value is called Kemeny's constant and is denoted by $K=K(M)$.
\end{theorem}
Since $K_x$ is independent of $x$, we can rewrite Kemeny's constant as 
$$
K=\sum_{x}\pi(x)K_x = \sum_{x,y\in\mathcal{X}}\pi(x)\pi(y)h_{xy} = \frac{1}{2}\sum_{x,y\in\mathcal{X}}\pi(x)\pi(y)(h_{xy}+h_{yx}).
$$
The sum $h_{xy}+h_{yx}$ is the expected number of steps it takes a Markov chain which starts at $x$ to reach $y$ and then $x$ again. This is called the commute time and is denoted by $c_{xy}:=h_{xy}+h_{yx}$. Note that $c$ is symmetric and $c_{xx}=0$ for every state $x$. As shown in \cite{wang_2017_kemeny}, this gives Kemeny's constant as
\begin{equation}\label{eq: definition Kemeny's constant via commute times}
K = \frac{1}{2}\sum_{x,y\in\mathcal{X}}\pi(x)\pi(y)c_{xy}.
\end{equation}
For Theorem \ref{th: main theorem}, we require two further assumptions on the Markov chain $M$. First, we assume that $M$ is loop-free\footnote{Without this assumption, the transition matrix cannot be represented by a Laplacian matrix (see Section \ref{section: simplex from M}).}, which means that $P_{xx}=0$ for every state $x$ and second, we assume that $M$ is reversible, which is defined by the condition $P_{xy}\pi(x)=P_{yx}\pi(y)$ for every pair of states $x$ and $y$. We collect the assumptions for further reference:
\begin{equation}\label{eq: assumptions on M}    \text{~$M$ is a discrete-time, finite, irreducible, aperiodic, reversible and loop-free Markov chain}.
\end{equation}
We will say that $M$ is a \eqref{eq: assumptions on M}-Markov chain if it satisfies these conditions.

\section{The simplex associated with a Markov chain}\label{section: simplex from M}
A $d$-dimensional simplex $S$ is the convex hull of $d+1$ affinely independent points in $\R^d$, which are called the vertices of $S$. A face of a simplex is the convex hull of a subset of its vertices, and every face is again a simplex. A one-dimensional face is called an edge and a $(d-1)$-dimensional face is called a facet.
\begin{theorem}\label{th: simplex embedding theorem}
Let $M$ be a \eqref{eq: assumptions on M}-Markov chain on $\mathcal{X}$. Then there exists an $(\vert\mathcal{X}\vert-1)$-dimensional simplex whose squared edge lengths are equal to the commute times of $M$.
\end{theorem}
Our proof combines the derivation of Chandra et al. \cite{chandra_1996_electrical} for the relation between commute time and effective resistance, with Fiedler's result that the effective resistance is the squared Euclidean distance between the vertices of a simplex \cite{fiedler_2011_matrices}. We first prove a technical lemma on certain matrices that can be associated with Markov chains. Let $L$ be the symmetric $\vert\mathcal{X}\vert\times\vert\mathcal{X}\vert$ matrix with diagonal entries $L_{xx}=\pi(x)$ and non-positive off-diagonal entries $L_{xy}=-\pi(x)P_{xy}$ for all $x\neq y$. This is known as the Laplacian matrix (of the weighted graph associated with $M$) and is a well-studied object in algebraic and spectral graph theory \cite{chung_spectral_1997}.
\begin{lemma}\label{lemma: Laplacians}
Let $L$ be the Laplacian matrix associated with a \eqref{eq: assumptions on M}-Markov chain on $\mathcal{X}$. Then there exists a matrix $L^\dagger$ which is the inverse of $L$ on $\spn(1)^\perp$ and which can be decomposed as $L^\dagger=V^TV$, where $V$ is a $(\vert\mathcal{X}\vert-1)\times\vert\mathcal{X}\vert$ matrix with affinely independent columns and $V1=0$.
\end{lemma}
\begin{proof}
Note that by construction $L_{xx}=-\sum_{y}L_{xy}$ for all $x$, which means that the rows of $L$ sum to zero. From the quadratic form determined by the Laplacian matrix
\begin{align}
f^TLf &= \sum_x f(x)\Big(\sum_{y\neq x}L_{xy}f(y)+L_{xx}f(x)\Big)\nonumber
\\
&= \sum_x f(x)\sum_{y\neq x}L_{xy}(f(y)-f(x))\nonumber
\\
&= \sum_{x,y\in\mathcal{X}} L_{xy}(f(y)-f(x))(f(x)-f(y))\nonumber
\\
&=-\sum_{x,y\in\mathcal{X}}L_{xy}(f(x)-f(y))^2 \geq 0,\label{eq: Laplacian quadratic form}
\end{align}
we find that $L$ is positive semidefinite. Furthermore, the quadratic form \eqref{eq: Laplacian quadratic form} is equal to zero if and only if $f(x)=f(y)$ whenever  $L_{xy}\neq 0$ or, equivalently, whenever $P_{xy}\neq 0$. By irreducibility of $M$, this implies that $f^TLf$ is zero if and only if $f$ is constant on $\mathcal{X}$. Thus $L$ is positive semidefinite with $\ker(L)$ spanned by the constant vector; in particular, $L$ is invertible when restricted to $\spn(1)^\perp$. This is a classical result, see for instance \cite{chung_spectral_1997}.

The Moore-Penrose pseudoinverse $L^\dagger$ of $L$ is the unique matrix which satisfies
$$
L^\dagger L L^\dagger = L^\dagger,\quad LL^\dagger L=L,\quad (LL^\dagger)^T = LL^\dagger, \quad (L^\dagger L)^T=L^\dagger L.
$$
This matrix always exists uniquely and its definition implies that $L^\dagger$ is positive semidefinite and has $\ker(L^\dagger)$ spanned by the constant vector. Furthermore, $L^\dagger$ is the inverse of $L$ in $\ker(L)^\perp=\operatorname{span}(1)^\perp$. Since $L^\dagger$ is positive semidefinite of rank $\vert\mathcal{X}\vert-1$, there exists a full-rank $\vert\mathcal{X}\vert\times(\vert\mathcal{X}\vert-1)$ matrix $V=[v_1~\dots~v_{\vert\mathcal{X}\vert}]^T$ such that $L^\dagger = V^TV$ and $V1=0$.

Now suppose for contradiction that the columns of $V$ are not affinely independent. Then $0=\sum_{i=2}^{\vert\mathcal{X}\vert} r_i(v_i-v_1)$ for some non-zero coefficients $r_2,\dots r_{\vert\mathcal{X}\vert}$. But then we can define new coefficients $\tilde{r}_i=r_i$ for $i\neq 1$ and $\tilde{r}_1=-\sum_{i=2}^{\vert\mathcal{X}\vert}\tilde{r}_i$ which sum to zero, i.e., $\tilde{r}\perp\spn(1)$, and satisfy $0 = \sum_{i=1}^{\vert\mathcal{X}\vert} \tilde{r}_i v_i$; in other words $\tilde{r}\in\ker(V)$. However, since $V$ has full rank and $V1=0$, we know that $\ker(V)=\spn(1)$ and thus $\tilde{r}\perp \ker(V)$. This contradiction proves that the columns of $V$ must be affinely independent and completes the proof.    
\end{proof}
\\~\\
We are now ready to prove Theorem \ref{th: simplex embedding theorem}, which associates a simplex to every \eqref{eq: assumptions on M}-Markov chain.

\begin{proof} \textbf{(of Theorem \ref{th: simplex embedding theorem})}
Let $M$ be a \eqref{eq: assumptions on M}-Markov chain. Since $M$ is irreducible, the hitting times $h$ are finite and well-defined. Fix a state $x$, then the hitting times to $x$ satisfy the recursion
$$
h_{tx} = \sum_{z\in\mathcal{X}}P_{tz}(1+h_{zx}) \text{~for all $t\in\mathcal{X}\backslash x$}.
$$
Multiplying both sides by $\pi(t)$ and using the fact that $\sum_z P_{tz}=1$, we obtain
\begin{equation}\label{eq: elementwise Laplacian}
\pi(t) = \sum_{z\in \mathcal{X}} -L_{tz}(h_{tx}-h_{zx}) \text{~for all $t\in\mathcal{X}\backslash x$}, 
\end{equation}
where we recall that $L_{tz}=-\pi(t)P_{tz}$ are off-diagonal entries of the Laplacian matrix. We now analyse this system of equations using matrix and vector notation. We consider the following functions on the states as vectors in $\mathbb{R}^{\vert\mathcal{X}\vert}$: the stationary distribution $\pi$, the constant vector $1$, the hitting times $h_x:t\mapsto h_{tx}$ to $x$ and the indicator function $\delta_x:y\mapsto 1$ iff $y=x$. Equation \eqref{eq: elementwise Laplacian} can then be written using the Laplacian matrix $L$ as
$$
Lh_x = \pi + \alpha_x\delta_x,
$$
where $\alpha_x$ is some constant which indicates that \eqref{eq: elementwise Laplacian} is determined except at the diagonal. Using $1^T\pi = 1$ and $1^TL=0$, which holds by construction of $L$, we find that this constant equals
$$
0 = 1^TLh_x = 1^T\pi + \alpha_x1^T\delta_x = 1 + \alpha_x\Rightarrow \alpha_x=-1.
$$
Since $x$ was an arbitrary state, we may repeat the derivation above for $y$ to obtain $Lh_y = \pi-\delta_y$. Combining the expressions for $x$ and $y$, this yields
\begin{equation}\label{eq: matrix Laplacian}
L(h_x-h_y) = \delta_y-\delta_x.
\end{equation}
Multiplying \eqref{eq: matrix Laplacian} from the left with $(\delta_y-\delta_x)^TL^\dagger$, where we recall that $L^\dagger$ is an inverse of $L$ in $\spn(1)^\perp$ as in Lemma \ref{lemma: Laplacians}, we find
$$
(\delta_y-\delta_x)^T(h_x-h_y) = (\delta_y-\delta_x)^TL^\dagger(\delta_y-\delta_x).
$$
With the definition of the commute times and using the decomposition $L^\dagger = V^TV$, we obtain
\begin{equation}\label{eq: commute time and edges}
c_{xy} = (\delta_y-\delta_x)^TL^\dagger(\delta_y-\delta_x) = (\delta_y-\delta_x)^TV^TV(\delta_y-\delta_x) = \Vert v_y-v_x\Vert_2^2.   
\end{equation}
Since the columns $v_1,\dots,v_{\vert\mathcal{X}\vert}$ of $V$ are $\vert\mathcal{X}\vert$ affinely independent points in $\mathbb{R}^{\vert\mathcal{X}\vert-1}$ by Lemma \ref{lemma: Laplacians}, they are the vertices of a ($\vert\mathcal{X}\vert-1$)-dimensional simplex. By \eqref{eq: commute time and edges}, the squared edge lengths of this simplex are equal to the commute times in $M$, which completes the proof.
\end{proof}

\section{Special points in the simplex and Kemeny's constant}\label{section: special points and main result}
The relation between the Markov chain $M$ and the simplex $S$ can be thought of as an embedding which maps the states of $M$ onto the vertices of $S$, and distributions on $\mathcal{X}$ to points in $S$. More generally, this embedding associates to every point in $\mathbb{R}^{\vert\mathcal{X}\vert-1}$ a unique unit-sum function on $\mathcal{X}$ which we call its coordinates\footnote{This is called the ``{barycentric coordinates}" in \cite{fiedler_2011_matrices}.}: the \textit{coordinates} of a point $p$ are denoted by $\hat{p}$ and are uniquely determined by $\sum_x \hat{p}(x)=1$ and $\sum_{x}\hat{p}(x) v_x=p$ where $v_x$ is the vertex associated with state $x$. 

We start with two lemmas that express distances between points and facets of the simplex in terms of their coordinates.
\begin{lemma}[{\cite[Cor. 1.4.14]{fiedler_2011_matrices}}]\label{lemma: distance from facets}
Let $M$ be a Markov chain with stationary distribution $\pi$ and associated simplex $S$. The distance from a point with coordinates $\hat{p}$ to the facet $S_{\mathcal{X}\backslash x}$ opposite the vertex associated with $x$ is equal to $\vert \hat{p}(x)/\sqrt{\pi(x)}\vert$.
\end{lemma}
\begin{lemma}[{\cite[Cor. 1.2.10]{fiedler_2011_matrices}}]\label{lemma: distances as quadratic forms}
Let $M$ be a \eqref{eq: assumptions on M}-Markov chain with associated simplex $S$ and commute times $c$. The squared distance between two points $a,b\in\mathbb{R}^{\vert\mathcal{X}\vert-1}$ with coordinates $\hat{a}$ and $\hat{b}$ is equal to
$$
\Vert a-b\Vert_2^2 = -\frac{1}{2}\sum_{x,y\in\mathcal{X}}c_{xy}[\hat{a}(x)-\hat{b}(x)][\hat{a}(y)-\hat{b}(y)].
$$
\end{lemma}
\begin{proof}
Let $V=[v_1~\dots~v_{\vert\mathcal{X}\vert}]^T$ be the matrix with vertices of the simplex $S$ as columns, such that $a=V\hat{a}$ and $b=V\hat{b}$. Since $c_{xy}=(v_x-v_y)^T(v_x-v_y)=v_x^Tv_x+v_y^Tv_y-2v_x^Tv_y$, the commute time matrix equals $C=1\nu^T+\nu 1^T-2V^TV$ where $\nu=\diag(V^TV)$. We then find
$$
\Vert a-b\Vert_2^2 = (\hat{a}-\hat{b})^TV^TV(\hat{a}-\hat{b})=(\hat{a}-\hat{b})^T[1\nu^T+\nu 1^T-\tfrac{1}{2}C](\hat{a}-\hat{b})=-\tfrac{1}{2}(\hat{a}-\hat{b})^TC(\hat{a}-\hat{b}),
$$
where we use that coordinates are unit-sum functions and thus $1^T(\hat{a}-\hat{b})=0$.
\end{proof}
\\
~
\\
We can now consider the two points of interest for the main result.
\begin{proposition}[{\cite[Cor. 1.4.13, Thm. 2.2.3]{fiedler_2011_matrices}}]\label{proposition: circumcenter and Lemoine point coordinates}
Let $M$ be a \eqref{eq: assumptions on M}-Markov chain with simplex $S$ and commute times $c$. 
\begin{itemize}
    \item The \textbf{circumcenter} $\gamma$ is the unique point at equal distance $R$ from all vertices of $S$. Its coordinates $\hat{\gamma}$ are determined by the equations $\sum_y c_{xy}\hat{\gamma}(y)=2R^2$ for every state $x$.
    \item The \textbf{Lemoine point} $\ell$ is the unique point at minimal total squared distance to the facets of $S$. Its coordinates equal $\hat{\ell}=\pi$.
\end{itemize}
\end{proposition}
\begin{proof} We start with the circumcenter. For every state $x$, the distance from the vertex $v_x$ (with coordinates $\hat{v_x}=\delta_x$) to the circumcenter $\gamma$ is equal to $R$. By Lemma \ref{lemma: distances as quadratic forms}, we find
$$
R^2 = d^2(v_x,\gamma) 
\,=\, -\frac{1}{2}[\delta_x - \hat{\gamma}]^TC[\delta_x - \hat{\gamma}] \,=\, \delta_x^T C\hat{\gamma} - \tfrac{1}{2}\hat{\gamma}^T C\hat{\gamma}.
$$
Since $R$ and $\tfrac{1}{2}\hat{\gamma}^TC\hat{\gamma}$ are independent of $x$, so must $\delta_x^TC\hat{\gamma}$ be, and thus $C\hat{\gamma}=\alpha\cdot 1$ for some $\alpha\in\mathbb{R}$. Introducing this expression back into the equation above, we obtain $\alpha$ as
$$
R^2=\delta_x^T C\hat{\gamma} - \tfrac{1}{2}\hat{\gamma}^T C\hat{\gamma} \iff R^2 = \alpha - \tfrac{1}{2}\alpha \iff \alpha=2R^2.
$$
We thus find that $\hat{\gamma}$ must satisfy $C\hat{\gamma}=2R^2\cdot 1$ which completes the first part of the Proposition.
\\
We continue with the Lemoine point and follow the derivation of Fiedler in \cite[Thm. 2.2.3]{fiedler_2011_matrices}. Let $p$ be a point in $\mathbb{R}^{\vert\mathcal{X}\vert-1}$. By Lemma \ref{lemma: distance from facets}, the total squared distance from $p$ to the facets of $S$ is equal to
$$
\sum_{x\in\mathcal{X}}d^2(p,S_{\mathcal{X}\backslash x}) = \sum_{x\in\mathcal{X}}\frac{\hat{p}(x)^2}{\pi(x)} = \sum_{x\in\mathcal{X}}\frac{\hat{p}(x)^2}{{\pi(x)}}\sum_{x\in\mathcal{X}}{\pi(x)} \geq \Big(\sum_{x\in\mathcal{X}}\hat{p}(x)\Big)^2 = 1.
$$
Here, we used the fact that $\pi$ and $\hat{p}$ are unit-sum vectors and the Cauchy--Schwarz inequality $\big(\sum_{x}u_x v_x\big)^2\leq \big(\sum_x u_x^2\big)\big(\sum_x v_x^2\big)$ with $u_x=\vert\hat{p}(x)\vert /\sqrt{\pi(x)}$ and $v_x=\sqrt{\pi(x)}$. Equality in the Cauchy--Schwarz bound is achieved if and only if $u_x=v_x$, and thus $\hat{p}(x)=\pi(x)$ in this case. In other words, the smallest total squared distance to all facets is $1$ and this is achieved by the point with coordinates $\pi$, as claimed. This completes the proof.
\end{proof}

We note that the Lemoine point always lies in the interior of the simplex $S$ because $\pi(x)>0$ for all $x$. On the other hand, the circumcenter lies inside the simplex if and only if $\hat{\gamma}(x)\geq 0$ for all $x$. In \cite[Ch. 6]{devriendt_2022_thesis} this condition is related to a notion of `nonnegative discrete curvature' for the weighted graph associated with simplex $S$. What does $\hat{\gamma}\geq 0$ imply for a Markov chain?

With Proposition \ref{proposition: circumcenter and Lemoine point coordinates} and Lemma \ref{lemma: distances as quadratic forms}, we can now prove the main result; this is Theorem \ref{th: main theorem-intro} from the introduction.
\begin{theorem}\label{th: main theorem}
In a finite, irreducible, aperiodic, reversible, loop-free Markov chain, Kemeny's constant $\kem$ is equal to
$$
\kem = R^2 - \Vert \gamma - \ell\Vert^2_2,
$$
where $\gamma$ is the circumcenter at distance $R$ from each vertex of $S$, the simplex whose squared edge lengths are given by the commute times of the Markov chain, and $\ell$ is the Lemoine point of $S$.
\end{theorem}
\begin{proof}
Let $M$ be a \eqref{eq: assumptions on M}-Markov chain with stationary distribution $\pi$, commute times $c$ and associated simplex $S$. The squared distance between the circumcenter $\gamma$ and the Lemoine point $\ell$ equals
\begin{align*}
\Vert \gamma-\ell\Vert_2^2 &= -\frac{1}{2}\sum_{x,y\in\mathcal{X}}c_{xy}[\hat{\ell}(x)-\hat{\gamma}(x)][\hat{\ell}(y)-\hat{\gamma}(y)] \quad\text{~(Lemma \ref{lemma: distances as quadratic forms})}
\\
&= -\frac{1}{2}\sum_{x,y\in\mathcal{X}}c_{xy}[\pi(x)-\hat{\gamma}(x)][\pi(y)-\hat{\gamma}(y)] \quad\text{~(Proposition \ref{proposition: circumcenter and Lemoine point coordinates})}
\\
&= \sum_{x,y\in\mathcal{X}}\left[-\tfrac{1}{2}c_{xy}\pi(x)\pi(y) + c_{xy}\pi(x)\hat{\gamma}(y) -\tfrac{1}{2}c_{xy}\hat{\gamma}(x)\hat{\gamma}(y)\right]
\\
&= -\frac{1}{2}\sum_{x,y\in\mathcal{X}}c_{xy}\pi(x)\pi(y) + 2R^2\sum_{x\in\mathcal{X}}(\pi(x)-\tfrac{1}{2}\hat{\gamma}(x)) \quad\text{(Proposition \ref{proposition: circumcenter and Lemoine point coordinates})}
\\
&= -K + R^2 \quad\quad\text{(Eq. \eqref{eq: definition Kemeny's constant via commute times})}.
\end{align*}
This completes the proof.
\end{proof}

\bibliographystyle{abbrv}
\bibliography{bibliography.bib}
\end{document}